\documentclass[a4paper]{scrartcl}
\usepackage[utf8]{inputenc}
\usepackage[russian, english]{babel}
\usepackage{amsmath}
\usepackage{amsthm}
\usepackage{amsfonts}
\usepackage{amssymb}
\usepackage{mathtools}
\usepackage[capitalize]{cleveref}
\usepackage{graphicx}
\usepackage[most]{tcolorbox}

\usepackage{tikz}
\usetikzlibrary{arrows}
\usetikzlibrary{external}
\addto\captionsrussian{}

\newtheorem*{theorem*}{Theorem}
\newtheorem{lemma}{Lemma}

\newtheorem*{corollary*}{Corollary}

\newtheorem*{proposition*}{Proposition}
\newtheorem{remark}{Remark}
\newtheorem*{remark*}{Remark}

\newtheorem*{definition*}{Definition}

\makeatletter
\DeclareRobustCommand{\cev}[1]{%
  \mathpalette\do@cev{#1}%
}
\newcommand{\do@cev}[2]{%
  \fix@cev{#1}{+}%
  \reflectbox{$\m@th#1\vec{\reflectbox{$\fix@cev{#1}{-}\m@th#1#2\fix@cev{#1}{+}$}}$}%
  \fix@cev{#1}{-}%
}
\newcommand{\fix@cev}[2]{%
  \ifx#1\displaystyle
    \mkern#23mu
  \else
    \ifx#1\textstyle
      \mkern#23mu
    \else
      \ifx#1\scriptstyle
        \mkern#22mu
      \else
        \mkern#22mu
      \fi
    \fi
  \fi
}

\makeatother

\newcommand{\revc}{\cev{c}}

\begin{document}

\title{On the combinatorics of circular codes}
\author{Aleksandr Serdiukov
\footnote{\texttt{serdyukov-alexander@ya.ru}, Laboratory of Continuous Mathematical Education, Saint Petersburg, Russia},
Andrei Smolensky
\footnote{\texttt{andrei.smolensky@gmail.com}, Saint Petersburg State University, Saint Petersburg, Russia}}
\date{}
\maketitle

\begin{abstract}
The present paper is devoted to the study of the combinatorics of 216 maximal $C^3$ circular codes --- a particular type of structure arising in the analysis of genomic sequences. Their circularity property is believed to be intimately connected to the protection against the reading frame shift in the process of RNA translation. We present some new observations concerning the internal structure of circular codes, which give a way to construct all of them in a relatively simple manner.
\end{abstract}

\section{Introduction}
In 1996 Arqu\`{e}s and Michel \cite{ffst} divided all non-monogenous trinucleotides into three parts $X_0$, $X_1$ and $X_2$ according to the reading frame they most often occur in inside the protein coding sequences of genes in both prokaryotic and eukaryotic organisms. Namely, the following codons most often occur in the non-shifted reading frame:
\begin{align*}
X_0 = \{ & \mathtt{AAC}, \mathtt{AAT}, \mathtt{ACC}, \mathtt{ATC}, \mathtt{ATT}, \mathtt{CAG}, \mathtt{CTC}, \mathtt{CTG}, \mathtt{GAA}, \mathtt{GAC}, \\
& \mathtt{GAG}, \mathtt{GAT}, \mathtt{GCC}, \mathtt{GGC}, \mathtt{GGT}, \mathtt{GTA}, \mathtt{GTC}, \mathtt{GTT}, \mathtt{TAC}, \mathtt{TTC} \}.
\end{align*}
They have noticed that this set enjoys particularly nice properties, namely, it is a maximal $C^3$ circular code (see below).

A set $\mathcal{C}\subset2^{\mathcal{B}^3}$ of $3$-letter words on the genetic alphabet $\mathcal{B}=\{\mathtt{A}, \mathtt{C}, \mathtt{G}, \mathtt{T}\}$ is called a circular code (CC for short) if any concatenation of its words written on a circle can be decomposed into concatenation in a unique way.
A circular code $\mathcal{C}$ is called maximal if it has the maximal possible number of elements, that is, $|\mathcal{C}|=20$.
A circular code $\mathcal{C}$ is called self-complementary if for every word $w=N_1N_2N_3\in\mathcal{C}$ its reverse complement $\revc(w)=c(N_3)c(N_2)c(N_1)$ is also a word from $\mathcal{C}$. Here, as usual, $c\colon \mathcal{B}\to\mathcal{B}$ is the complementarity map:
\[ c(\mathtt{A})=\mathtt{T},\quad c(\mathtt{T})=\mathtt{A},\quad c(\mathtt{C})=\mathtt{G},\quad c(\mathtt{G})=\mathtt{C}. \]
Denote by $\alpha\colon\mathcal{B}^3\to\mathcal{B}^3$ the cycle permutation $(123)$, that is, $\alpha(N_1N_2N_3)=N_3N_1N_2$. A circular code $\mathcal{C}$ is called a $C^3$ circular code if $\revc(\mathcal{C})=\mathcal{C}$ and $\alpha(\mathcal{C})$ is also circular (then so is $\alpha^2(\mathcal{C})$ and $\revc(\alpha(\mathcal{C}))=\alpha^2(\mathcal{C})$).

Later all $216$ maximal $C^3$ circular codes have been identified \cite{cfcc,clas20c,identification} by means of a extensive computer calculations. The aim of the present paper is to study the interrelations between them.

In what follows by a circular code we will always mean a maximal $C^3$ circular code.

Two distinct circular codes will be called \emph{variative} if their intersection is of maximal possible size ($18$ words or $9$ pairs of complementary codons).

\section{The internal structure of a circular code}
Let $w\in\mathcal{B}^3$ be a codon. We say that that the shape of $w$ is $\mathtt{XXY}$, $\mathtt{XYY}$, $\mathtt{XYX}$ or $\mathtt{XYZ}$ if this codon can be obtained from it by a substitution of bases for $\mathtt{X}$, $\mathtt{Y}$ and $\mathtt{Z}$.
Let $\mathcal{C}$ be a circular code, define the shape-sets $sh(\mathcal{C}, \mathtt{XXY})$, $sh(\mathcal{C}, \mathtt{XYY})$, $sh(\mathcal{C}, \mathtt{XYX})$ and $sh(\mathcal{C}, \mathtt{XYZ})$ to be the sets of codons of $\mathcal{C}$ of the respective shape.
\begin{lemma}
$sh(\mathcal{C}, \mathtt{XYX})$ has $0$, $2$ or $4$ elements.
\end{lemma}
\begin{proof}
If $\mathcal{C}$ contains a word $w$ of shape $\mathtt{XYX}$, then it must also contain its complement, which is distinct from $w$, so the size of $sh(\mathcal{C},\mathtt{XYX})$ is even.

Note that since a circular code can not simultaneously contain words of the form $N_1N_2N_1$ and $N_2N_1N_2$, the first two letters of $w$ are non-complement to each other. This shows that there are at most $8$ codons of shape $\mathtt{XYX}$ that can appear in a circular code, but no more that a half of them simultaneously.
\end{proof}
\begin{lemma}
\label{8XYZ}
$sh(\mathcal{C}, \mathtt{XYZ})$ has exactly $8$ elements.
\begin{proof}
There are $24$ words of shape $\mathtt{XYZ}$, and they are naturally split into $8$ orbits under the $C_3$ action by circular shifts.

If $\mathcal{C}$ contains more than $8$ words of shape $\mathtt{XYZ}$ then at least two of them would be a circular shifts of each other, and this breaks the circularity property \cite{transformations}.

Since for every non-monogenous codon $w$ the three circular shifts $w$, $\alpha(w)$ and $\alpha^2(w)$ are evenly distributed between $\mathcal{C}$, $\alpha(\mathcal{C})$ and $\alpha^2(\mathcal{C})$, the result follows.
\end{proof}
\end{lemma}
\begin{remark}
$\mathcal{C}$ cannot simultaneously contain $N_1N_1N_2$ and $N_2N_2N_1$.
\end{remark}
\begin{proof}
If both $N_1N_1N_2, N_2N_2N_1\in\mathcal{C}$, then $N_1N_2N_1, N_2N_1N_2\in\alpha^2(\mathcal{C})$, which is impossible by the circularity of $\alpha^2(\mathcal{C})$.
\end{proof}
Since $|\mathcal{C}|=20$ and $\revc(sh(\mathcal{C}, \mathtt{XXY})=sh(\mathcal{C}, \mathtt{XYY})$, the size of $sh(\mathcal{C}, \mathtt{XXY})$ is $4$, $5$ or $6$ depending on the size of $sh(\mathcal{C}, \mathtt{XYX})$.
\begin{lemma}
Let $N_1,N_2\in\mathcal{B}$ be non-complementary, and $M_1,M_2\in\mathcal{B}$ be the two remaining bases. The code $\mathcal{C}$ contains at least one of $N_1N_1N_2$, $N_2N_2N_1$, $M_1M_1M_2$ and $M_2M_2M_1$.
\end{lemma}
\begin{proof}
If $|sh(\mathcal{C}, \mathtt{XXY})|>4$, the result is obviuos from the numerical considerations. The size of $sh(\mathcal{C}, \mathtt{XXY})$ equals $4$ only if $|sh(\mathcal{C}, \mathtt{XYX})|=4$.

If $\mathcal{C}$ contains $NMN$, then $MNM\in\alpha^i(\mathcal{C})$, $i=1$ or $2$. Then $c(M)c(N)c(M)\in\alpha^{-i}(\mathcal{C})$. If $i=1$, then $c(M)c(M)c(N)\in\mathcal{C}$. If $i=2$, then $c(N)c(M)c(M)\in\mathcal{C}$ and hence $MMN\in\mathcal{C}$. Varying $NMN$ over the four elements of $sh(\mathcal{C}, \mathtt{XYX})$, one gets the result.
\end{proof}
\begin{remark}
For any $N\in\mathcal{B}$ the code $\mathcal{C}$ contains exactly one of $NNc(N)$ and $c(N)c(N)N$.
\end{remark}
\begin{proof}
Since $Nc(N)N, c(N)Nc(N)\notin\mathcal{C}$, either $Nc(N)N$ or $c(N)Nc(N)$ lies in $\alpha^2(\mathcal{C})$, hence either $NNc(N)$ or $c(N)c(N)N$ is in $\mathcal{C}$.
\end{proof}
\begin{remark}
$sh(\mathcal{C}, \mathtt{XYX})$ is determined completely by $sh(\mathcal{C}, \mathtt{XXY})$.
\end{remark}
\begin{proof}
Indeed, for $N,M\in\mathcal{B}$ the codon $NMN$ lies in $\mathcal{C}$ if and only if $NNM$ and $MNN$ do not.
\end{proof}
	
\section{The action of $D_4$}
It was noted in \cite{transformations} that there is a natural action of dihedral group $D_4$ on the circular codes. Namely, it is induced by the action of $D_4$ by square symmetries on the vertices of the square labeled by $\mathcal{B}$:
\begin{center}
\includegraphics{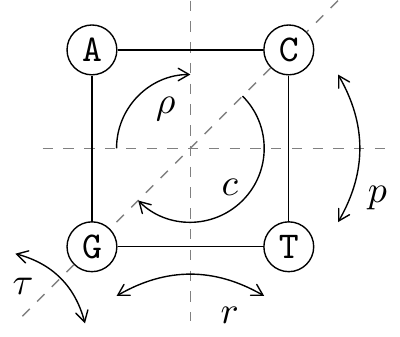}
\end{center}
Here the central reflection is the complementarity map $c$, the reflections with respect to the horizontal and vertical axes are the purine-pyrimidine duality $p$ and the keto-amine duality $r$. We also denote by $\rho$ the $90^\circ$ rotation and by $\tau$ the reflection with respect to the diagonal $\mathtt{CG}$.

The letter-wise action of $D_4$ on $\mathcal{B}^3$ induces an action on circular codes, that has the following properties.

\begin{lemma}\label{lemma:d4-cc-free}
$D_4$ action on circular codes is free.
\end{lemma}
\begin{proof}
Let $\mathcal{C}$ be a circular code. We will show that stabilizer $\operatorname{Stab}_{D_4}(\mathcal{C})$ is trivial.

Since $\mathcal{C}$ cannot simultaneously contain both $N_1N_1N_2$ and $c(N_1)c(N_1)c(N_2)$, $c\notin\operatorname{Stab}(\mathcal{C})$, hence also $\rho,\rho^{-1}\notin\operatorname{Stab}(\mathcal{C})$.

$\mathcal{C}$ contains exactly one of $\mathtt{AAT}$ and $\mathtt{TTA}$. But $\tau(\mathtt{AAT})=\mathtt{TTA}$, so $\tau\notin\operatorname{Stab}(\mathcal{C})$. A similar argument for $\mathtt{CCG}$ and $\mathtt{GGC}$ shows that the other diagonal reflection is not in the stabilizer.

There is at least one of $\mathtt{AAG}$, $\mathtt{GGA}$, $\mathtt{CCT}$ and $\mathtt{TTC}$ in $\mathcal{C}$. But $p$ swaps $\mathtt{AAG}$ and $\mathtt{GGA}$, which cannot both lie in $\mathcal{C}$, and the same goes for $\mathtt{CCT}$ and $\mathtt{TTC}$. This shows that $p\notin\operatorname{Stab}(\mathcal{C})$. Similar considerations for $\mathtt{AAC}$, $\mathtt{CCA}$, $\mathtt{GGT}$ and $\mathtt{TTG}$ imply $r\notin\operatorname{Stab}(\mathcal{C})$.
\end{proof}
\begin{remark}
This action can not be extended to the letter-wise action of $S_4$.
\end{remark}
\begin{proof}
For a set $X\subset 2^{\mathcal{B}^3}$ that satifies the circularity property and is of maximal size the self-complementarity is preserved by $\sigma\in S_4$ if and only if $\sigma$ commutes with $c$, and the centralizer of $c$ in $S_4$ is the $D_4$ as described above.
\end{proof}
Note that the shape of a codon is not changed by the action of $D_4$, and so for each shape $S$ one can define the action of $D_4$ on the set of shape-sets
\[ sh_S = \{ sh(\mathcal{C}, S) \mid \mathcal{C}\ \text{is a circular code} \}. \]
\begin{remark}
The action of $D_4$ on $sh_{\mathtt{XXY}}$ is free.
\end{remark}
\begin{proof}
Indeed, the proof of \cref{lemma:d4-cc-free} only uses the structure of $sh(\mathcal{C}, \mathtt{XXY})$.
\end{proof}
\begin{lemma}
$sh_{\mathtt{XYX}}$ has $9$ elements and splits into three orbits of length $1$, $4$ and $4$.
\end{lemma}
\begin{proof}
There are $4$ shape-sets of size $2$, and the shape-sets of size $4$ are obtained as the unions of shape-sets of size $2$. But of all 6 possible size $4$ unions two cannot be shape-sets for they simultaneously contain codons of the form $N_1N_2N_1$ and $N_2N_1N_2$.
\end{proof}
\begin{lemma}
$sh_{\mathtt{XYZ}}$ has $24$ elements and splits into $4$ orbits of sizes $4$ and $8$.
\end{lemma}
\begin{proof}
Let $\mathcal{C}$ be a circular code. Since $|sh(\mathcal{C}, \mathtt{XYZ})|=8$, we can describe the shape-set for $\mathtt{XYZ}$ in more details. Namely, for any choice of three bases $N_1$, $N_2$ and $N_3$ there are exactly two words on these letters in $sh(\mathcal{C}, \mathtt{XYZ})$.

The two orbits of length $4$ are formed by the shape-sets that have even number of distinct bases for the first letter of their elements. They are stable under the action of certain reflections. Namely, shape-sets with only two distinct first letters are stable under $p$ or $r$, while those having four disting first letters are stable under $\tau$ or $\rho^{-1}\tau\rho$. In the latter case the are only two distinct middle letters.

The two orbits of length $8$ are formed by shape-sets having three distinct first letters. The difference between these two orbits is the number of pairs of codons of the form $N_1N_2N_3$ and $N_3N_2N_1$, the shape-sets in one orbit have two such pairs, the shape-sets in the other have none.
%
\end{proof}
The four orbits classified in the proof above will be denoted as follows:
\begin{itemize}
\item Type $\mathrm{I}$ is the length $4$ orbit of $p$- or $r$-invariant shape-sets;
\item Type $\mathrm{II}$ is the length $4$ orbit of $\rho$- or $\rho^{-1}\tau\rho$-invariant shape-sets;
\item Type $\mathrm{III}$ is the length $8$ orbit of shape-sets with two pairs of the form $N_1N_2N_3$ and $N_3N_2N_1$.
\item Type $\mathrm{IV}$ is the length $8$ orbit of shape-sets with no pairs of the form $N_1N_2N_3$ and $N_3N_2N_1$;
\end{itemize}

We now list the restrictions on the possible combinations of shape-sets for $\mathtt{XXY}$ and $\mathtt{XYZ}$.

For every possible shape-set for $\mathtt{XXY}$ there corresponds at most one element from each orbit of $sh_{\mathtt{XYZ}}$, with the sole exception described below.

The shape-set of type $\mathrm{I}$ corresponding to $\mathcal{C}$ has $N$ as one of the first letters for its codons if $NNc(N)\in\mathcal{C}$.

The shape-set of type $\mathrm{II}$ corresponding to $\mathcal{C}$ has $N$ for the middle letter if $\mathcal{C}$ contains a codon of the form $MNM$.

The shape-set of type $\mathrm{III}$ corresponding to $\mathcal{C}$ has $N$ for the middle letter of the symmetric pair of codons if $\mathcal{C}$ contains a codon of the form $MNM$. In particular, there in no such codon in $|sh(\mathcal{C}, \mathtt{XXY})|=6$.

The shape-set of type $IV$ has a base that appears five times as the first letter of its codons, the two other first letters appear two and one times. Such a shape-set $F$ corresponding to $\mathcal{C}$ is chosen as follows:
\begin{itemize}
\item If there are three distinct first letters $N, c(N), M$ in the codons of $sh(\mathcal{C}, \mathtt{XXY})$, then $F$ is the type $\mathrm{IV}$ shape-set with $M$ as the most frequent first letter and $N, c(N)$ as the other two.
\item If there are only two distinct first letters $N, M$ in the codons of $sh(\mathcal{C}, \mathtt{XXY})$, then their numbers of appearances are $(2,2)$, $(3,1)$ or $(3,2)$. In the first case the type $\mathrm{II}$ shape-set is chosen by the same considerations as above, but done for the frist letters of $sh(\mathcal{C}, \mathtt{XYY})$. In the second case the type $\mathrm{II}$ shape-set has $M$ for the most frequent first letter. In the third case there are two possible type $\mathrm{II}$ shape-sets, having either $N$ or $M$ for the most frequent first letter.
\end{itemize}

There is no corresponding shape-set for $\mathtt{XYZ}$ of a certain type if:
\begin{itemize}
\item $|sh(\mathcal{C}, \mathtt{XXY})|=4$, there are two codons of the form $NNM$ and $NNc(M)$ and all other codons of this shape-set do not have $N$ as the first letter. In this case there is no type $\mathrm{I}$ shape-set.
\item $sh(\mathcal{C}, \mathtt{XXY})$ is of the form $\{ NNc(N),MMc(M), NNM, MMc(N), c(M)c(M)c(N) \}$ for non-complementary $N$ and $M$. Then $\mathcal{C}$ cannot have a type $\mathrm{I}$ shape-set for $\mathtt{XYZ}$.
\item There are four different middle letters in $sh(\mathcal{C}, \mathtt{XYX})$. There is no corresponding type $\mathrm{II}$ shape-set.
\item There are three distinct first letters for the codons of $sh(\mathcal{C}, \mathtt{XXY})$ and one of them appear $3$ times, while the two others only one. Then both types $\mathrm{II}$ and $\mathrm{IV}$ are forbidden.
\item Type $\mathrm{III}$ shape-set is impossible in all cases when there is no corresponding shape-set of type $\mathrm{I}$ or no shape-set of type $\mathrm{II}$.
\end{itemize}

This list of restrictions is complete in the sense that it allows one to construct all $216$ circular codes by simply combining shape-sets by the rules described above. Modulo the $D_4$ action on $sh_{\mathtt{XXY}}$ each choice of a corresponding element from $sh_{\mathtt{XYZ}}$ gives a representative for the $D_4$ orbit on the set of circular codes. Indeed, there are $27$ possible combinations and no two are equivalent under this action.


\section{The graph of variative circular codes}
We conclude the paper with the following remark on the graph of variative circular codes.

Consider the graph $\Gamma$ with circular codes for vertices and edges connecting variative codes. The group $D_4$ acts on it by graph automorphisms. All $216$ vertices are divided into $27$ orbits under this action, each having $8$ elements. The interrelation between these orbits appears quite cumbersome, however, we have been able to figure out one unexpected regularity.

Namely, consider the quotient graph $\Gamma/D_4$, that is, its vertices are $D_4$-orbits and two orbits are connected if some of their members are connected in $\Gamma$. It turns out that the degree of each vertex in $\Gamma/D_4$ equals the degree of its elements as vertices of $\Gamma$. Note that this does not necessary hold for an arbitrary graph acted on by a group. For example, consider a square with diagonals and the action of $C_2$ by a reflection:
\begin{center}
\includegraphics[]{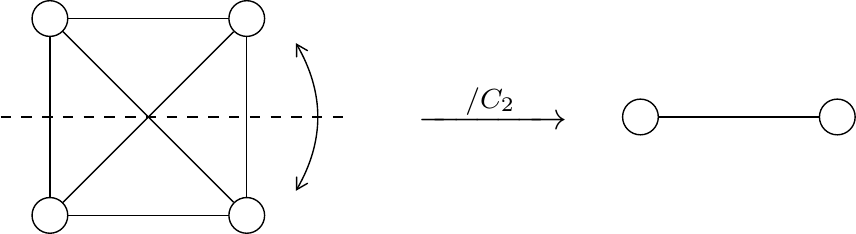}
%
%
\end{center}
This example also show that this property does not necessary hold for other natural definitions of a quotient graph (where multiple edges and loops are allowed).

\end{document}